\documentclass[12pt]{amsart}

\usepackage{latexsym}
\usepackage{amsmath}
\usepackage{amssymb}
\usepackage{amscd}
\usepackage{multicol}
\usepackage[cmtip,matrix,arrow]{xy}
\usepackage{amsmath}

\numberwithin{equation}{section}

\newtheorem{theorem}{\bf Theorem}[section]
\newtheorem{corollary}[theorem]{\bf Corollary}
\newtheorem{proposition}[theorem]{\bf Proposition}
\newtheorem{lemma}[theorem]{\bf Lemma}
\newtheorem{definition}[theorem]{\bf Definition}
\newtheorem{definition-theorem}[theorem]{\bf Theorem-Definition}
\newtheorem{remark}[theorem]{\bf Remark}
\newtheorem{ex}[theorem]{Example}

\def\bR{\mathbb{R}}
\def\bC{\mathbb{C}}

\def\bQ{\mathbb{Q}}

\def\t{\mathfrak{t}}
\def\g{\mathfrak{g}}
 
\def\quott({/\! /}

\def\g{\mathfrak{g}}
\def\t{\frak{t}}

\def\s{\sigma}
\def\t{\tau}

\address[O.~Goertsches]{Department of Mathematics and Computer Science\\  Philipps-Universit\"at Marburg\\  Germany}
\address[S.~Hagh Shenas Noshari]{Institut f\"ur Geometrie und Topologie\\  Universit\"at Stuttgart\\  Germany}
\address[A.-L.~Mare]{Department of Mathematics and Statistics\\ University of Regina\\ Canada}

\email[]{goertsch@mathematik.uni-marburg.de}

\email[]{sam.noshari@mathematik.uni-stuttgart.de}

\email[]{mareal@math.uregina.ca}

\title[Equivariant cohomology of hyperpolar actions]
{On the equivariant cohomology of hyperpolar actions on symmetric spaces 
}

\author[O.~Goertsches]{Oliver Goertsches}
\author[S.~Hagh Shenas Noshari]{Sam Hagh Shenas Noshari}
\author[A.-L.~Mare ]{Augustin-Liviu Mare }
\date{\today}

\begin{document}
\begin{abstract}
We show that the equivariant cohomology of any hyperpolar action of a compact and connected Lie group
on a symmetric space of compact type is a Cohen-Macaulay ring. This generalizes some results previously obtained
 by the authors.
\end{abstract}

%\address{{\rm Liviu Mare}, Department of Mathematics and Statistics, University of Regina,  Regina, Saskatchewan} \email{mareal@math.uregina.ca}

\maketitle
\section{Introduction}\label{intro}

When investigating smooth group actions on manifolds, an insightful topological invariant one associates to it
is its equivariant cohomology. (The coefficient ring throughout this paper is $\bQ$, unless otherwise specified.) 
This remains an effective tool even in cases when concrete descriptions are hard to obtain.
Geometric features  of the action are reflected by algebraic properties of its equivariant
cohomology. Equivariant formality is an example of such a property, 
which is relevant due to its simple definition  (see Sect.~\ref{sec:general})
and to the numerous situations when it is satisfied, such as Hamiltonian
actions of compact Lie groups on compact symplectic manifolds (cf.~\cite{Ki})
or isotropy actions of compact symmetric spaces (see \cite{Go}). 
A larger class consists of the so-called Cohen-Macaulay actions,
which are defined in Sect.~\ref{sec:general} below.
Their relevance in the context of equivariant topology was 
noticed by several authors already  in the  70s and early 80s: for example, M.~F.~Atiyah used the Cohen-Macaulay condition
to study  equivariant $K$-theory for torus actions in \cite{Atiyah} and his ideas 
were adapted to equivariant cohomology by G.~E.~Bredon in \cite{Bredon}; see also the paper \cite{Duflot} by J.~Duflot for actions of discrete groups. 
 More recently, Lie group actions whose equivariant cohomology satisfies the requirement  above have been considered in  \cite{FranzPuppe2003},
before being thoroughly
investigated in  \cite{Go-Ro} and \cite{Go-To}. 
 They are natural generalizations of equivariantly formal group actions to the more general situation when
all isotropy subgroups may have their ranks strictly smaller than the rank of the acting group.
Cohen-Macaulay actions are worth studying due to their
interesting features: for example, just like in the equivariantly formal case, 
a compact Lie group  action  and the induced action of a maximal torus satisfy simultaneously
this requirement; furthermore, Cohen-Macaulay actions are
characterized by the exactness of a certain Atiyah-Bredon long sequence;
for more details we refer to Sect.~\ref{cmsec} and the appendix below. 
Also, this condition  is satisfied in numerous concrete situations, for example 
when the cohomogeneity of the action is zero or one, as it was pointed out in \cite{Go-Ma}.

Cohomogeneity-one actions are special cases of hyperpolar actions. 
Recall  that an isometric action of a compact connected Lie group 
  on a Riemannian manifold is {\it polar} if there exists a submanifold which is
  intersected by each orbit of the action,  the  orbit being orthogonal to the submanifold
  at each intersection point. If the submanifold can be chosen to be flat relative to
  the induced Riemannian metric, we say that the action is {\it hyperpolar}.
  In this paper we will be interested in the case when  the manifold on which the group acts
  is a Riemannian symmetric space of compact type,
  that is, a quotient $G/H$, where $G$ is  a compact connected semisimple Lie group 
  and $H$ a closed subgroup such that $G^\sigma_0 \subset H \subset G^\sigma$, where $\sigma$ is 
  an involutive automorphism of $G$, $G^\sigma$ its fixed point set and $G^\sigma_0$ the identity component
  of the latter group. Polar actions on compact symmetric spaces have  been extensively investigated by many authors,
  see  for instance \cite{HPTT1}, \cite{HPTT2}, \cite{Ko1}, \cite{Ko2}, \cite{Ko3}, and \cite{Ko}.
  We will prove the following:
  
  \begin{theorem}\label{thm:realmain}
  Any hyperpolar action of a compact connected Lie group on a 
  symmetric space of compact type is Cohen-Macaulay.
  \end{theorem}
  
  The proof relies essentially on the classification of the actions
  mentioned in the theorem, which was obtained  
  by A.~Kollross in \cite{Ko}. Before stating this result, 
  we need to describe an important class of examples  
  of hyperpolar actions on the symmetric space $G/H$ mentioned above.  
  Let $\tau$ be another involutive automorphism of $G$
  and consider its fixed point set, $G^\tau$, along with the identity component 
  $K:=G^\tau_0$.  The action of $K$ on
   $G/H$ by left translations turns out to be hyperpolar, see
   \cite[Ex.~3.1]{HPTT1}.
   It is known under the generic name of a  {\it Hermann action}, after %the name   of 
   R.~Hermann, who first investigated this situation in \cite{Her}
    (note that originally, in \cite{Her}, the group $G^\tau$ was assumed
   to be connected).  According to Kollross' theorem \cite{Ko}, 
    any hyperpolar action is orbit equivalent
  to a direct product of actions of one of the following types:
  transitive, of cohomogeneity one, or Hermann (for the notion of orbit equivalence, see Sect.~\ref{oea}). 
  We already know that any action of one of the first two types
  is Cohen-Macaulay, see \cite[Cor.~1.2]{Go-Ma}. Thus, to prove Thm.~\ref{thm:realmain}, there are two
  steps to be performed: show first that any Hermann action is Cohen-Macaulay and second that the Cohen-Macaulay property is preserved 
  under 
  orbit equivalence. 
  The first goal is achieved by the following result:
  
  \begin{theorem} \label{main} If $G$, $K$, and $H$ are as above, then the (Hermann) action of $K$ on $G/H$ by left translations is Cohen-Macaulay.
   \end{theorem}
  
  A big part of the paper is devoted to the proof of this theorem, see Sect.~\ref{semisimple}
 (also note that the result remains true even if $G$ is not necessarily semisimple, see Rem.~\ref{nonsemi}). 
 It is worth recalling at this point a related result, obtained by the first-named author of this paper  in \cite{Go},
 which says that the $K$-action on $G/K$ by left translations is equivariantly formal;
 recently, a conceptually different proof has been  obtained by the second-named author in \cite{No}. 
  Thm.~\ref{main} above is a generalization of this (taking into account Prop.~\ref{characef} below). Pairs $(G,K)$ with the property
  that the action of $K$ on $G/K$ is equivariantly formal 
   have also been investigated in \cite{Go-No}, \cite{Ca2}, \cite{Ca1}, \cite{Ca-Fo}, and again \cite{No}.

The second of the aforementioned goals is to show that the Cohen-Macaulay property is preserved
under orbit equivalence. This turns out to be true:
see Thm.~\ref{orbit}, which represents the main result of Sect.~\ref{oea}. 
The  proof of Thm.~\ref{thm:realmain} is also given in full detail  in that section.    

\noindent{\bf Acknowledgement.} We would like to thank the referee  for numerous valuable comments
on a previous version of this paper.
   
\section{General considerations}\label{sec:general}

\subsection{The action of $G$ on $M$}\label{cmsec}
Throughout this subsection $G$ will always be a  compact and connected Lie group
which acts smoothly on a closed manifold $M$,
although many of the results presented below hold in a  larger generality.  
To any such action   
one attaches the equivariant cohomology $H^*_G(M)$.
This can be
defined as the usual cohomology of the Borel construction $EG\times_G M$,
where $EG\to BG$ is the classifying principal bundle of $G$.
This construction originally belongs to A.~Borel \cite{Bo} and it became
gradually an important tool in the theory of transformation
groups by work of M.~F.~Atiyah and R.~Bott \cite{At-Bo}, D.~Quillen \cite{Quillen}, and many
others. The book \cite{AlldayPuppe} by C.~Allday and V.~Puppe is a useful reference for this topic. 
%For an introduction to this theory we refer to \cite[Ch.~III]{Hs} or \cite[Appendix C]{GGK}. 
We confine ourselves to mentioning the basic fact  that $H^*_G(M)$ has a canonical structure of an 
$H^*(BG)$-algebra. 

In case $H^*_G(M)$  is free as an $H^*(BG)$-module, we say that the $G$-action is {\it equivariantly formal}.
Here is a well-known class of examples of such actions (cf.~\cite[Thm.~6.5.3]{GS}):

\begin{ex} \label{exodd}
{\rm If $H^{\rm odd}(M)=0$ then any $G$-action on $M$ is equivariantly formal.}
\end{ex}

The following criterion will be useful later, see \cite[p.~46]{Hs}.

\begin{proposition}\label{criteqf}
Let $T$ be  a torus acting on a closed manifold $M$.

(a) The total Betti number of
the fixed point set $M^T$ is at most equal to the total Betti number of $M$.

(b) The two numbers mentioned above are equal if and only if the $T$-action is equivariantly formal.
\end{proposition}

In what follows we will be looking in more detail at the algebraic 
structure of $H^*_G(M)$. Relative to its structure of an $H^*(BG)$-module,
it is a positively graded module over a positively graded ring. The latter, $H^*(BG)$, 
is a polynomial algebra on generators of even degrees,
 and thus a commutative, graded, Noetherian ring. 
It is also a *local ring, in the sense that it has a unique graded ideal
which is maximal among all graded ideals (namely $H^{>0}(BG)$).
 The theory of graded modules over *local Noetherian rings  is nicely treated in  \cite[Sect.~1.5]{BrunsHerzog}. 
 Let us just  briefly recall that to each such module  one can associate 
its depth, that is, the maximal length of a regular sequence of elements in the maximal graded ideal. One can also associate its Krull dimension, which is the Krull dimension of the ring modulo the annihilator of the module. In general, the latter number is at least equal to the depth. 
We say that the module is Cohen-Macaulay if it is equal to 0 or its Krull dimension is equal to its depth. For a more detailed account of this topic, 
 suitable for applications to equivariant cohomology, we  refer to \cite[Sect.~5]{Go-To}.

 \begin{definition}\label{def:cm} A smooth action $G\times M \to M$ is called {\rm Cohen-Macaulay}
 if $H^*_G(M)$, considered as a module over $H^*(BG)$, is Cohen-Macaulay. 
 \end{definition} 
 
A systematic study of this type of actions was undertaken in \cite{Go-Ro} and \cite{Go-To}
and we will often rely on results obtained there. 
However, some caution is necessary, since in the two references above 
the coefficient field for cohomology is $\bR$, whereas in this paper it is $\bQ$.
We will next show that the two seemingly different Cohen-Macaulayness conditions
are in fact equivalent. This will be done by using the following general result:%, for which we refer to \cite[Thm.~2.1.7]{BrunsHerzog}:

\begin{proposition}\label{commalg}
Let $R$ and $S$ be two Noetherian *local rings  and $\varphi : R \to S$ a  ring homomorphism 
which is homogeneous of degree 0 and maps the maximal graded ideal ${\mathfrak m}$ of $R$ into the maximal graded ideal
${\mathfrak n}$ of $S$.  Assume that both $R/{\mathfrak m}$ and $S/{\mathfrak n}$ are fields.
 Let $A$  be a finitely generated non-zero $R$-module and
$B$ a finitely generated non-zero $S$-module, which is flat over $R$. Then:

(a) ${\rm depth}_S \ \! A\otimes_R B = {\rm depth}_R \ \! A + {\rm depth}_S \ \! B/{\mathfrak m}B$;
%where the depth is always relative to ${\mathfrak m}$ and ${\mathfrak n}$  respectively;  

(b) $\dim_S A\otimes_R B = \dim_R  A + \dim_S  B/{\mathfrak m}B$;

(c) The module $A\otimes_R B$ is Cohen-Macaulay over $S$ if and only if $A$ and $B/{\mathfrak m}B$
are Cohen-Macaulay over $R$ and $S$, respectively.
\end{proposition} 

\begin{proof} In the special case when $R$ and $S$ are just local rings, both (a) and (b) are standard facts in commutative algebra,
cf.~e.g.~Prop.~1.2.16 and Thm.~A.11 in \cite{BrunsHerzog}. 
The general case can be reduced to this special situation by using localization at 
the maximal graded ideals. One takes into account that if $I$ is a graded ideal of $R$ then
$ {\rm grade} \ \! (I, A)={\rm grade} \ \! (I_{\mathfrak m}, A_{\mathfrak m})$,
see \cite[Prop.~1.5.15 (e)]{BrunsHerzog}: this takes care of the three terms involved
in equation (a). For (b), one uses three times the general formula $\dim A=\dim A_{\mathfrak m}$,
for which we refer to the proof of Prop.~5.1 in \cite{Go-To} (note that the argument found there uses
 the assumption that $R/{\mathfrak m}$ is a field).   
\end{proof}

\begin{corollary}
A group action $G\times M \to M$ is Cohen-Macaulay in the sense of Def.~\ref{def:cm}
if and only if $H^*_G(M; \bR)$ is Cohen-Macaulay as a module over $H^*(BG; \bR)$.
\end{corollary}  

\begin{proof} One  applies Proposition \ref{commalg} by taking $R=H^*(BG)$, $S=H^*(BG; \bR)$, $A=H^*_G(M)$, and $B=H^*(BG;\bR)$,
 in each case with
the obvious module structure; also, take $\varphi :  H^*(BG)\to H^*(BG; \bR)$
the  inclusion map.
\end{proof}

\begin{remark} \label{pair} {\rm 
The same argument shows that if $N$ is a $G$-invariant subspace of $M$, then
the $G$-equivariant cohomologies $H^*_G(M, N)$ and $H^*_G(M, N; \bR)$ as modules
over $H^*(BG)$, respectively $H^*(BG; \bR)$, have the same depth, the same Krull dimension,
and are thus simultaneously
Cohen-Macaulay.}  
\end{remark}
 
 An immediate observation is that any equivariantly formal action is Cohen-Macaulay.
 The converse implication is in general not true, as one can easily see
 in concrete situations (for example, by \cite[Prop.~2.6]{Go-Ma}, any transitive action is Cohen-Macaulay,
 without being in general equivariantly formal). However, the following result is helpful in this context,
 see \cite[Prop.~2.5]{Go-Ro}:
 
 \begin{proposition}\label{characef}
 A smooth action $G\times M \to M$ is equivariantly formal if and only if it is Cohen-Macaulay
 and there exists at least one point in $M$ whose isotropy subgroup has rank equal to
 ${\rm rank } \ \! G$.
\end{proposition}

For later use, we also mention:

\begin{proposition}\label{maxtor} Let $G\times M \to M$ be a smooth action and $T\subset G$  an arbitrary maximal torus. 

(a) The $G$-action on $M$ is equivariantly formal if and only if so is the induced $T$-action.  

(b) The $G$-action on $M$ is Cohen-Macaulay if and only if so is the induced $T$-action. \end{proposition}

Item (a) is a standard result, see for instance \cite[Prop.~C.26]{GGK}.  Item (b) is the content of \cite[Prop.~2.9]{Go-Ro}.  It is worth pointing out in this context that if $H\subset G$ is an arbitrary closed subgroup, then the property of being equivariantly formal is preserved when passing from 
$G$ to $H$: indeed, on the one hand, the $G$-action is equivariantly formal if and only if the canonical map
$H^*_G(M)=H^*(E\times_G M) \to H^*(M)$ is surjective (cf., e.g., \cite[Prop.~2.1]{Go-Ma}), and on the other hand the previous map factors as
$H^*_G(M) \to H^*_H(M) \to H^*(M)$.   However,  it is possible for the $G$-action to be Cohen-Macaualy and the $H$-action not to be like that, see \cite[Ex.~2.13]{Go-Ro}.

From the previous proposition one can deduce:

\begin{corollary} \label{dirprod} If  two smooth actions on closed manifolds  are Cohen-Macaulay then their direct product is Cohen-Macaulay as well.
\end{corollary}

\begin{proof} Let the two group actions be 
$G_i \times  M_i \to M_i$ and let $T_i \subset G_i$ be maximal tori, 
where $i=1,2$. By Prop.~\ref{maxtor} above, the restricted $T_i$-action on $M_i$
is Cohen-Macaulay. 
By \cite[Rem.~2]{Go-To}, the restricted $T_i$-action on $M_i$
 is Cohen-Macaulay if and only if there exists a subtorus $S_i \subset T_i$ of rank equal to the minimal dimension
 of a $T_i$-orbit on $M_i$, which acts locally freely on $M_i$
  such that the induced action of $T/S_i$ on the orbit space $M/S_i$ is equivariantly formal. 
  But then $S_1 \times S_2$ acts locally freely on $M_1 \times M_2$, has rank equal to the minimal dimension
 of a $T_1\times T_2$-orbit on $M_1\times M_2$ and the induced action of 
  $T_1/S_1 \times T_2/S_2$ on $M_1/S_1 \times M_2/S_2$ is equivariantly formal (note that equivariant formality is preserved under taking the direct product). 
  Thus the $T_1\times T_2$-action on $M_1 \times M_2$ is Cohen-Macaulay, and so is the $G_1\times G_2$-action.  
 \end{proof}
 
 We will need an equivalent characterization of the Cohen-Macaulay condition above, this time
  exclusively in terms of the ring structure of $H^*_G(M)$. 
 %The latter is a finitely generated ring,  see for instance \cite[Thm.~2.1]{Quillen}.
The latter is in general not a commutative ring, but it is nevertheless  graded commutative,
in the sense that $x\cdot y =(-1)^{({\rm deg} x)({\rm deg} y)}y\cdot x$, 
for all homogeneous $x, y \in H^*_G(M)$. Although a self-contained treatment of graded commutative rings is not 
 immediately available  in the literature,
 there is no essential difference relative to the (usual) commutative case. 
A systematic and thorough study of graded commutative rings has been undertaken  by M.~Poulsen in Appendix A
of his Master Thesis \cite{Po}. 
For instance, by using the remark following
Prop.~A.5 in \cite{Po} and also taking into account that 
 $H^*_G(M)$ is finitely generated as an algebra over its degree zero component 
  $H^0_G(M)\simeq \bQ$ (see for instance \cite[Cor.~2.2]{Quillen}),
   we deduce that $H^*_G(M)$ is a Noetherian ring.
%Its degree zero component is $H^0_G(M)\simeq \bR$.
For graded commutative Noetherian rings of the type $R=\oplus_{i\ge 0}R^i$
whose degree zero component $R^0$ is a field, one can see in \cite{Po} that
 the concepts of 
Krull dimension, depth (relative to the ideal $\oplus_{i>0}R^i$),
and Cohen-Macaulayness can be defined in the same way 
 as for commutative rings. This enables us to
prove the following result. It is obtained by adapting 
\cite[Prop.~12, Sect.~IV.B]{Se} to our current set-up.

\begin{lemma} {\rm (J.-P.~Serre)} \label{serre}
Let $R=\oplus_{i\ge 0}R^i$ 
 and $S=\oplus_{i\ge 0}S^i$
be  two graded commutative Noetherian rings with $R^0=S^0=\bQ$
 and $\varphi: R\to S$ a homomorphism of graded rings which makes 
 $S$ into an $R$-module which is finitely generated. 
 Let also $A$ be a finitely generated $S$-module.
 Then $A$ is Cohen-Macaulay as an $S$-module if and only if so is $A$
 as an $R$-module.      
\end{lemma}

\begin{corollary}\label{cmcrit}
The $G$-action on $M$ is Cohen-Macaulay if and only if the ring
$H^*_G(M)$ is Cohen-Macaulay.
\end{corollary}

   \subsection{The action of $K$ on $G/H$} 
    Let $G$ be again a compact and connected Lie group and let
$K, H\subset G$ be closed subgroups. In this section we list some results concerning the
three group actions mentioned in the following proposition.

\begin{proposition}\label{crit} If  $K$ and  $H$ are  connected,
the following assertions are equivalent:
\begin{itemize}  
\item[(a)] the action of  $K$ on $G/H$ by left translations is Cohen-Macaulay;
\item[(b)] the action of  $H$ on $G/K$ by left translations is Cohen-Macaulay;
\item[(c)] the action of  $H\times K$ on $G$ given by 
\begin{equation}\label{hgk}(h,k)\cdot g=hgk^{-1},\end{equation}
is Cohen-Macaulay.
\end{itemize}
\end{proposition}

\begin{proof}
For the equivalence between (b) and (c) observe that 
 we have  the ring isomorphisms  $$ H^*_{H\times K}(G)\simeq H^*_H(G/K),$$
 due to the fact that the  factor $K$ of the product $H\times K$ acts freely on $G$.
 The equivalence between (a) and (c) can be proved in a similar way, although
 this time one also needs to take into account that the $H\times K$-action in the lemma 
 is equivalent to the group action $(H\times K)\times G \to G$,
 $(h, k)\cdot g = kgh^{-1}.$
 The result now follows from Cor.~\ref{cmcrit}.
\end{proof}

In the light of item (c) above, the following proposition is useful. It is a direct consequence
of a certain Sullivan model obtained by 
 J.~D.~Carlson and C.-K.~Fok in \cite[Sect.~3.1]{Ca-Fo} (see also \cite[Sect.~8.8.3]{Ca1} or \cite[end of Sect.~2]{Ca3}), which describes the equivariant cohomology  of actions
on $G$ of
type (\ref{hgk}), the acting group being a general subgroup of
$G\times G$.  
 Denote by ${\mathfrak g}, {\mathfrak k}, {\mathfrak h}$ the Lie algebras of $G$, $K$, and $H$.

\begin{proposition}{\rm (J.~D.~Carlson and C.-K.~Fok)} \label{carlson} If $K$ and $H$ are connected, then the equivariant cohomology of the action
of  $H\times K$ on $G$ given by $(h,k).g=hgk^{-1}$ depends only on 
${\mathfrak g}$, ${\mathfrak h}$, ${\mathfrak k}$, 
and the inclusions ${\mathfrak h}\hookrightarrow {\mathfrak g}$ and $ {\mathfrak k}\hookrightarrow  {\mathfrak g}$. 
\end{proposition}

Finally, we mention a result that shows how to deal with the situation when  $H$ is not connected.

\begin{proposition}\label{redcon}
Assume $K$ is connected and let $H_0$ be the identity component of $H$.
If the $K$-action on $G/H_0$ is Cohen-Macaulay, then so is the $K$-action on $G/H$.
\end{proposition} 

\begin{proof}
Use the characterization given by Prop.~\ref{crit}, (c).  Note that $H_0\times K$ is the identity component of $H\times K$,
hence 
$H^*_{H\times K}(G)\simeq H^*_{H_0\times K}(G)^{H/H_0}.$
By hypothesis,  $H^*_{H_0\times K}(G)$ is  Cohen-Macaulay as a module
over the ring $H^*(BH_0 \times BK)$, thus also over its subring $H^*(BH\times BK)=H^*(BH_0 \times BK)^{H/H_0}$,
since the  ring is a finitely generated module over the subring and one can use Lemma \ref{serre}. 
By a standard averaging argument, the invariant module  $H^*_{H_0\times K}(G)^{H/H_0}$
can be realized as a direct summand of $H^*_{H_0\times K}(G)$ in the category of 
$H^*(BH \times BK)$-modules, hence it is Cohen-Macaulay.  
\end{proof}

\section{Hermann actions}\label{semisimple}
In this section we will prove Thm.~\ref{main}. 
The meaning of $G$, $H$, and $K$ is as in  Sect.~\ref{intro}.
In particular, the Lie group $G$ is  compact, connected, and semisimple
(the case when $G$ is not necessarily semisimple is discussed in Rem.~\ref{nonsemi}).  
%Due to Propositions \ref{crit}, \ref{carlson}, and \ref{redcon}, we do not lose any generality by assuming that $G$ is simply connected and only prove Thm.~\ref{main} in this special situation.  Under this assumption, the groups $H$ and $K$ mentioned in the introduction are equal to the fixed point sets of the involutive automorphisms $\s$ and $\t$ respectively. 
Due to Propositions \ref{crit} and  \ref{redcon}, we only need to show that the equivariant cohomology ring of the $H_0 \times K$-action on $G$ is 
Cohen-Macaulay. By Prop.~\ref{carlson}, we lose no generality by assuming that $G$ is simply connected, since otherwise we can replace the triple 
$G$, $H_0$, $K$ with the universal cover of $G$  
along with its two connected and closed subgroups that are locally isomorphic to $H_0$ and $K$ respectively.  
Under the aforementioned assumption, the fixed point sets of $\sigma$ and $\tau$ are connected,
 see \cite[Thm.~1.8, Ch.~VII]{He}, and thus equal to $H_0$ and $K$ respectively. 

\subsection{The case when $G$ is simple}\label{sec:simple}

   We will prove Thm.~\ref{main} under the assumption that $G$ is simple and simply connected,
   which is valid throughout the whole subsection. The induced automorphisms of $\g$ will also be denoted by $\s$ and $\t$. 
In general, we will not make any notational distinction between an automorphism of $G$ and the
induced Lie algebra automorphism.

   \begin{lemma} \label{torus} Assume $G$ is not equal to ${\rm Spin}(8)$. If ${\rm rank} \ \!  K \le  {\rm rank} \ \! H$, then a maximal torus
   in $K$  is group-conjugate with a subgroup of $H$.
   \end{lemma} 
       \begin{proof} If one of $\sigma$ or $\tau$  is an inner automorphism, the claim in the lemma is clear by \cite[Thm.~5.6, p.~424]{He}.   
From now on we will assume that both $\sigma$ and $\tau$ are outer automorphisms of $G$.       
       Let $T_K\subset K$ be an arbitrary maximal torus.
       There is a unique maximal torus $T$ in $G$ such that $T_K\subset T$;
      furthermore, $T$ is $\tau$-invariant and there is a Weyl chamber $C$ in ${\mathfrak t}:={\rm Lie}(T)$ which is
       $\tau$-invariant as well
       (see \cite[Prop.~3.2, p.~125]{Lo}). Let $c: G \to G$ be an inner automorphism
       such that the torus $S:=c(T)$ contains $T_H$, the latter being a maximal torus in $H$.
       As before, there exists a Weyl chamber $C'\subset c({\mathfrak t})$ 
       which is invariant under $\sigma$. On the other hand, $c\tau c^{-1}$ 
       is an involutive automorphism of $G$ which leaves  $S$ invariant;
       it even leaves the chamber $c(C)$ (inside the Lie algebra of $S$)  invariant.
       But the chambers $c(C)$ and $C'$ are conjugate under the Weyl group of $(G,S)$; that is, there exists  an inner automorphism, say $c'$,
        such that $$c'(S) =S \ {\rm and} \ c'(C') = c(C).$$
       We now compare %the actions of 
       the involutions $c\tau c^{-1}$ and $c'\sigma c'^{-1}$ of $G$:
       they both leave the torus $S$ invariant, and along with it, its Lie algebra and the chamber $c(C)$ inside it:
       both are realized in terms of permuting the simple roots that determine the chamber, the permutation being
       necessarily a Dynkin diagram automorphism.
       
       Since $G$ is different from ${\rm Spin}(8)$, there is a unique such (involutive) permutation which is not the identity map.  
       But none of the automorphisms $c\tau c^{-1}$  and $c'\sigma c'^{-1}$ is inner, thus by composing them
       and then restricting the result to the chamber
       $c(C)$, one gets the identity map;  as an automorphism of $G$, this composition must consequently be an inner automorphism
       $c_g$ defined by some $g\in S$. That is,
       $$c\tau c^{-1}=c'\sigma c'^{-1} c_g.$$
       %\footnote{Ein Automorphismus von $G$ deren Einschr\"ankung auf $S$ gleich zur Identit\"a  ist, muss vom Typ $c_g$ sein, wobei $g\in T$. Warum genau ist es so?}  
       Notice that the automorphism in the left hand side of the previous equation leaves $c(T_K)\subset S$ pointwise fixed.
       Consequently, $c'\sigma c'^{-1}$ does the same.
       In other words, $c(T_K)$ is contained in the fixed point set of $c'\sigma c'^{-1}$,
       which is just $c'(H)$. This finishes the proof.
             \end{proof}

 \noindent{\it Proof of Thm.~\ref{main} in the case $G$ simple.} First assume that $G$ is different from ${\rm Spin}(8)$.
By Prop.~\ref{crit}, we may assume that ${\rm rank} \ \!  K \le  {\rm rank} \ \! H$. Thus, by Lemma \ref{torus}, there exists
 $T\subset K$ a maximal torus and $g_0\in G$ such that $g_0Tg_0^{-1}\subset H$.
 The actions $K \times G/H\to G/H$ and 
 $g_0Kg_0^{-1} \times G/H\to G/H$ given by left translation are equivalent, relative to the map
 $$(K, G/H)  \to (g_0Kg_0^{-1}, G/H), \quad (k, gH) \mapsto (g_0kg_0^{-1}, g_0gH).$$
 Thus it is sufficient to show that the $g_0Kg_0^{-1}$-action on $G/H$ is Cohen-Macaulay.
 Equivalently, by means of Prop.~\ref{maxtor} above, that the action of $g_0Tg_0^{-1}$ on $G/H$ is so.
 But this is clear, because, by the main result in \cite{Go}, the action of
 $H$ on $G/H$ is equivariantly formal. 
 
If $G={\rm Spin}(8)$, then, as in the proof of Lemma \ref{torus}, we can assume that both $\sigma$ and $\tau$ 
 are outer automorphisms. It is known that the group of outer automorphisms of ${\rm Spin}(8)$ is isomorphic 
 to the symmetric group $\Sigma_3$ on three letters and the fixed point set of any order-two outer automorphism 
 is isomorphic to ${\rm Spin}(7)$, giving rise to a copy of the symmetric space ${\rm Spin}(8)/{\rm Spin}(7) = S^7$ (cf.~\cite[Sect.~2]{Ad1}). Thus Thm.~\ref{main} is a consequence of Lemma \ref{odd} below.
 $\square$
 
 \begin{lemma}\label{odd} Any smooth action of a compact connected Lie group on a closed manifold which is a rational
 cohomology  sphere is Cohen-Macaulay. 
 \end{lemma}
 
 \begin{proof} If the dimension of the sphere is even, the lemma is a consequence of Ex.~\ref{exodd}. To deal with the remaining situation, we use 
 Prop.~\ref{maxtor} (b), which enables us to only consider  the action of a torus 
 $T$ on $X$, where $X$ is a closed manifold with $H^*(X)\simeq H^*(S^{2n+1})$ as vector spaces, for some $n \ge 0$.
 In particular, the total Betti number $\dim H^*(X)$ is equal to 2. In case the fixed point set $X^T$
  is non-empty, the latter is a union of closed submanifolds of $X$.   
  By Prop.~\ref{criteqf},  $\dim H^*(X^T)$  
is thus a strictly positive  number at most equal to 2. On the other hand, the Euler-Poincar\'e characteristics of $X$ and $X^T$
are equal, see e.g.~\cite[Thm., (4)]{Kobayashi} or \cite[Thm.~9.3]{Go-Zo}; thus $\dim H^*(X^T)$ is an even number. We conclude that the latter number is equal to 2, hence the $T$-action on $X$  is equivariantly formal and consequently Cohen-Macaulay.  

 Let us now consider the case when the set $X^T$ is empty.
 There exists a one-dimensional subtorus $S\subset T$ whose action on $X$ is locally free.
 One can compute the cohomology of the orbit space $X/S$ by means of the following version of
 the Gysin sequence (cf.~\cite[Lemma 2.2]{Co}):
 $$\ldots \to H^j(X/S) \to H^{j+2}(X/S) \to H^{j+2}(X) \to H^{j+1}(X/S) \to \ldots .$$
 It follows that $H^*(X/S)\simeq H^*(\bC P^n)$ by an isomorphism of vector spaces. 
 Consequently, by Ex.~\ref{exodd}, the canonical action of $T/S$ on $X/S$ is equivariantly formal,
 hence, by Cor.~\ref{cmcrit}, the ring $H^*_{T/S}(X/S)$ is Cohen-Macaulay. 
 To conclude the proof, it only remains to notice that we have the ring isomorphism 
 $$H^*_T(X)\simeq H^*_{T/S}(X/S).$$
 \end{proof}
 
 \subsection{The case when $G$ is not simple}
 From now on, $G$ is just simply connected, not necessarily  being simple. 
Our main tool in dealing with this situation is the following result, see \cite[Prop.~18]{Ko}:

\begin{proposition}{\rm (A.~Kollross)}\label{andreas} If $G$ is simply connected then the $H$-action on $G/K$ is
a direct product of actions of one of the following types:
\begin{itemize}
\item[(i)] the action of $H'\times L^{n-1} \times K'$ on $L^n$ defined by
\begin{align*}
{}&(h, g_1, \ldots, g_{n-1}, k)\cdot(x_1, \ldots, x_n)\\{}&=
(hx_1g_1^{-1}, g_1x_2g_2^{-1}, \ldots, g_{n-2}x_{n-1}g_{n-1}^{-1}, g_{n-1}x_n k^{-1}),\end{align*}
\item[(ii)] the action of $H'\times L^{n-1}$ on $L^{n-1}\times L/K'$ defined by
\begin{align*}
{}&(h, g_1, \ldots, g_{n-1})\cdot(x_1, \ldots, x_{n-1}, x_nK')\\{}&=
(hx_1g_1^{-1}, g_1x_2g_2^{-1}, \ldots, g_{n-2}x_{n-1}g_{n-1}^{-1}, g_{n-1}x_n K'),\end{align*}
\item[(iii)] the action of $L^{n-1}$ on $H'\backslash L\times L^{n-2}\times L/K'$ defined by
\begin{align*}{}&(g_1,\ldots, g_{n-1})\cdot (H'x_1, x_2, \ldots, x_{n-1}, x_n K')
\\{}&=(H'x_1g_1^{-1}, g_1x_2g_2^{-1}, \ldots, g_{n-2}x_{n-1}g_{n-1}^{-1}, g_{n-1}x_n K'),\end{align*}
\item[(iv)] the action of $L^n$ on $L^n$ defined by
\begin{align*}{}&(g_1, \ldots, g_n)\cdot(x_1,\ldots, x_n)\\{}&
=(g_1x_1g_2^{-1}, g_2x_2g_3^{-1}, \ldots, g_{n-1}x_{n-1}g_n^{-1}, g_n x_n\alpha(g_1)^{-1}),\end{align*}
\end{itemize}
where $L$ is a simply connected, simple and compact Lie group, $H',K'\subset L$ are fixed points of 
involutions of $L$,  $\alpha$ is an outer or trivial automorphism of $L$, and $n$ is an arbitrary integer,
at least equal to 1 in cases (i), (ii), and (iv), respectively to 2 in case (iii).
\end{proposition} 

For the reader's convenience, here are a few details concerning this result.

\begin{remark} {\rm (a) Proposition \ref{andreas} above is just a consequence of the result actually proved in \cite{Ko}. 
Namely, we are here in the special case when $G$ is simply connected. Thus our $H$-action on $G/K$ is of the Hermann type in the sense of 
\cite[Def.~15]{Ko} (see the definition of locally symmetric subgroups at the beginning of \cite[Sect.~2]{Ko} and also recall the well-known fact  that the identity component of the isometry group of the symmetric space $G/K$ is just $G$). Under the hypothesis mentioned above, $G/K$ is simply connected as well, hence the $H$-action on $G/K$ is not only locally conjugate to one of the types (i)-(iv), as the result in \cite{Ko} says, but genuinely conjugate. 

(b) We also wish to explain the meaning of the number $n$ in the proposition above. It comes from the structure of $G/K$. 
For example, in case (i) the pair $(G, K)$ is equal to $(L \times \cdots \times L, \Delta(L)\times \cdots \times 
\Delta(L))$, where the first direct product has $2n$ factors, the second has $n$, and $\Delta(L) :=\{(g, g) | g\in L\}$.     
The cases (ii)-(iv) are left as an exercise to the reader.}
\end{remark}

In view of Cor.~\ref{dirprod}, to prove Thm.~\ref{main} for $G$ semisimple it is sufficient to show that the actions (i)-(iv) are 
Cohen-Macaulay.

Let us start with  (i). We use an inductive argument. 
Note that the first factor $L$ in $H'\times L^{n-1}\times K'$ acts freely  on $L^{n}$, the orbit space
being diffeomorphic to $L^{n-1}$ via
$$L^{n}/L \to L^{n-1}, \quad (x_1, x_2, \ldots,  x_{n}) \mapsto (x_1 x_2, x_3 \ldots,  x_{n}).$$
Thus it is sufficient to prove that the action of $H'\times L^{n-2}\times K'$ on $L^{n-1}$ given by
\begin{align*}
{}&(h, g_2, \ldots, g_{n-1}, k)\cdot(x_2, \ldots, x_n)\\{}&=
(hx_2g_2^{-1}, g_2x_3g_3^{-1}, \ldots, g_{n-2}x_{n-1}g_{n-1}^{-1}, g_{n-1}x_n k^{-1}),\end{align*}is Cohen-Macaulay.
We continue the procedure and gradually drop out the $L$-factors in $H'\times L^{n-1}\times K'$  until we finally obtain the  action of $H'\times K'$ on $L$ given by $(h, k)\cdot x=hxk^{-1}$.
But this action is Cohen-Macaulay by the result already proved in Subsect.~\ref{sec:simple}
(see also Prop.~\ref{crit}).  

To deal with (ii),  we start by modding out the action of the first factor $L$ in $H'\times L^{n-1}$, which is clearly a free action. 
In this way, we reduce the problem to  showing that the action of  $H'\times L^{n-2} $ on
$L^{n-2}\times L/K'$ described by
\begin{align*}
{}&(h, g_2, \ldots, g_{n-1})\cdot(x_2, \ldots, x_{n-1}, x_nK')\\{}&=
(hx_2g_2^{-1}, g_2x_3g_3^{-1}, \ldots, g_{n-2}x_{n-1}g_{n-1}^{-1}, g_{n-1}x_n K'),\end{align*}is Cohen-Macaulay. We continue the procedure until we are led to the action of
$H'$ on $L/K'$ given by left translations. 
Again, this action is Cohen-Macaulay by the result we proved in Subsect.~\ref{sec:simple}.

Similarly, in case (iii) we reduce the problem to the action of
$L$ on $H'\backslash L \times L/K'$ described by $g\cdot (H'x_1, x_2K')=(H'x_1g^{-1}, gx_2K')$.
The  map $$H'\backslash L \times L/K'\to L/H'\times L/K', \quad (H'x_1, x_2K')\mapsto (x_1^{-1}H', K'x_2)$$
is an $L$-equivariant diffeomorphism. This allows us to change our  focus to  the action of $\Delta(L):=\{(g,g)\mid g\in L\}$
on $(L\times L)/(H'\times K')$ by left translations. By Prop.~\ref{crit}, this  is 
Cohen-Macaulay if and only if so is the action of $H'\times K'$ on $(L\times L)/\Delta(L)$.
But the latter is just the action of $H'\times K'$ on $L$ given by
$(h, k)\cdot x=hxk^{-1}$, which was discussed in Subsect.~\ref{sec:simple}.

 As about (iv), the recursive procedure already used in each of the previous cases 
 leads us to the action of $L$ on itself given by $g\cdot x=gx\alpha(g)^{-1}$. 
 It was proved by Baird in \cite[p.~212]{Ba} (cf.~also \cite[p.~58]{Sieb}) that all isotropy groups of this action
have the same rank. 
By \cite[Cor.~4.3]{Go-Ro}, the action is thus Cohen-Macaulay.

\begin{remark}\label{nonsemi} {\rm Thm.~\ref{main} holds even when $G$
is not necessarily semisimple. To prove this,  consider a finite cover of $G$ of the type $T\times G_s$, where
$T$ is a torus and $G_s$ is compact, connected, and simply connected.
In view of Propositions~\ref{crit}, \ref{carlson}, and \ref{redcon} it is sufficient to
consider the case when $G$ is equal to such a direct product and $H$ is the identity 
component of $G^\tau$. 
But then both $\sigma$ and $\tau$ leave the factors $T$ and $G_s$ invariant.
Their fixed point sets split as direct products of subgroups of $T$ and $G_s$ respectively;
that is, $H_0=T_1\times H_s$ and 
$K=T_2\times K_s$, where $T_1, T_2\subset T$ are subtori and $H_s, K_s \subset G_s$. 
The $H_0\times K$-action on $G$ can be described as follows:
$$\left((t_1, h), (t_2, k)\right). \left(t, g\right) = \left(t_1tt_2^{-1}, hgk^{-1}\right).$$
This is the direct product of the following two actions:
\begin{align*}{}&(T_1 \times T_2) \times T \to T, \  (t_1, t_2)\cdot t =t_1tt_2^{-1}\\
{}& (H_s \times K_s)\times G_s \to G_s, \ (h, k)\cdot g = hgk^{-1}.
\end{align*}
In view of Cor.~\ref{dirprod}, it is sufficient to observe that both of them are Cohen-Macaulay: 
for the first factor, one observes that the kernel of the $T_1 \times T_2$-action on $T$ is isomorphic 
to $T_1\cap T_2$, and an arbitrary direct complement of its  identity component acts locally freely;
  for the second factor, one uses Thm.~\ref{main}. } 
\end{remark}

%By Prop.~\ref{characef} above a smooth action of a compact connected  Lie group on a closed manifold is equivariantly formal if and only if it is Cohen-Macaulay  and at least one of the isotropy subgroups has maximal rank.  We  deduce as follows: 
   
   %\begin{corollary}\label{equivf} Let $G, K,$ and $H$ be like in Thm.~\ref{main}.  The action of $K$ on $G/H$ is equivariantly formal if and only if %there exists  a maximal torus in $K$
%which is conjugate with a subgroup of $H$. \end{corollary}

We conclude the section with a  remark concerning  equivariant formality of Hermann actions:

\begin{remark}
{\rm  Let $G, K$, and $H$ be as in Thm.~\ref{main}. The action of $K$ on $G/H$ is equivariantly formal if and only if a maximal torus 
in $K$ is conjugate with a subgroup of $H$ (note that the latter condition is equivalent to the existence of a fixed point for the action of the maximal torus in $K$,
which is necessary for the equivariant formality of the $K$-action; the other implication is proved by invoking the main result of \cite{Go} as in the proof of Thm.~\ref{main} in the case $G$ simple). In 
particular, ${\rm rank} \ \! K \le {\rm rank } \ \! H$. We note that the latter condition alone is not sufficient for equivariant formality. Take for instance two Dynkin diagram involutions of ${\rm Spin}(8)$ which are not conjugate with each other.
Their fixed point sets, $H$ and $K$ respectively, are both isomorphic to ${\rm Spin}(7)$. 
It turns out that the action of $H$ on ${\rm Spin}(8)/K\simeq S^7$ by left translations is transitive, with isotropy subgroups
isomorphic to the exceptional compact Lie group of type $G_2$, see e.g.~\cite[Thm.~3]{Va}. Thus this action is not equivariantly formal.
It is interesting to notice, however, that if $G$ is simple and simply connected,
$G\neq {\rm Spin}(8)$, then the condition ${\rm rank} \ \! K \le {\rm rank } \ \! H$ is sufficient
for the $K$-action on $G/H$ to be equivariantly formal: this follows from Lemma \ref{torus} above. } 
\end{remark}

\section{Orbit equivalent actions}\label{oea}

%\subsection{A general result}
We start by proving the following two lemmata, whose relevance will become clear immediately:
\begin{lemma}\label{prop:orbit}
Let $M$ be a closed manifold and $K, K'$ two compact and connected Lie groups
that act on $M$ such that $K\subseteq K'$ and $Kp=K'p$ for all $p\in M$.
Then the $K$-action on $M$ is Cohen-Macaulay if and only if so is the $K'$-action.
\end{lemma} 

\begin{proof}
Let $b$ denote the maximal rank of an isotropy subgroup of the $K$-action
and $M_{b,K}$ the subspace of $M$ consisting of all points whose isotropy group
has rank equal to $b$. In the same way, to the action of $K'$ one assigns
the number $b'$ and the subspace $M_{b', K'}$. 
For any $p\in M$ we have $Kp=K'p$. In general, if $G$ is a compact  Lie group 
and $U \subset G$  a closed subgroup,
then  the  rank
$\chi\pi(G/U)$ of the homogeneous space $G/U$ is a homotopy invariant, which 
 turns out to be equal to the difference  ${\rm rank } \ \! U - {\rm rank } \ \! G$
(see \cite[Def.~Sect.~1.1 and Sect. 4.2]{All2}).  In our situation, 
for any $p\in M$ we have $Kp = K'p$, and hence the homogeneous spaces 
$K/K_p$ and $K'/K'_p$ are diffeomorphic, where $K_p$ and $K'_p$ are the 
isotropy subgroups at $p$. 
We deduce \begin{equation}\label{rankk}{\rm rank } \ \! K - {\rm rank } \ \! K_p={\rm rank } \ \! K' -
{\rm rank } \ \! K'_p.\end{equation}
(We note that the same argument was used in \cite[Prop.~7]{Go-Po}; we could also just apply the latter proposition to the $K$- and $K'$-actions 
on $Kp=K'p$ to obtain the same conclusion.) Eq.~(\ref{rankk}) implies:
\begin{equation}\label{mbh} M_{b, K} = M_{b', K'}.\end{equation} 
Pick  maximal tori
$T$ in $K$ and $T'$ in $K'$ such that $T\subset T'$. 
The minimal dimension of a $T$-orbit in $M$ is 
${\rm rank} \ \! K - b$, cf.~e.g.~\cite[Lemma 4.1]{Go-Ro}.
Similarly, the minimal dimension of a $T'$-orbit in $M$ is 
${\rm rank} \ \! K' - b'$. But
the two aforementioned numbers are the $K$-, respectively $K'$-coranks of 
any point in $M_{b,K}$ respectively $M_{b',K'}$, hence,
by eq.~(\ref{mbh}), they are equal. 
Thus, there exists a subtorus $S\subset T$ of rank
equal to the two numbers which acts locally freely on
$M$.

The sets $M_{b, T}$ and $M_{b', T'}$ defined in the same way
as before are non-empty, clearly contained in $M_{b,K}$ and
$M_{b', K'}$ respectively. The $K$-action on $M_{b, K}$
is Cohen-Macaulay, see \cite[Cor.~4.3]{Go-Ro}.   
Consequently, this time by Prop.~\ref{maxtor} (b) and Cor.~\ref{cmcrit}, $H^*_T(M_{b,K})$ is a
 Cohen-Macaulay ring. Since $S\subset T$ acts locally freely,
 the latter ring is isomorphic to $H^*_{T/S}(M_{b,K}/S)$. 
 
 On the other hand, the $T/S$-action on $M_{b,K}/S$ admits  points that are fixed.
 Namely, they are orbits of the form $Sp$ such that $Sp=Tp$,
 which implies that \begin{equation}\label{fix}{\rm corank}_T \ \! T_p = {\rm rank} \ \! S.\end{equation}
But if $p\in M$ satisfies the latter condition,
then $Sp$ is a connected and closed submanifold of $Tp$, of the same dimension as the latter,
hence $Sp=Tp$. Thus condition (\ref{fix}) characterizes the fixed points.
Since ${\rm rank} \ \! S = {\rm rank} \ \! T -b$, that condition is actually equivalent
to $p\in M_{b, T}$. We have actually proved:
\begin{equation}\label{movers}(M/S)^{T/S}=(M_{b,K}/S)^{T/S}=M_{b,T}/S.\end{equation}
  
From the previous considerations, the $T/S$-action on $M_{b,K}/S$ is equivariantly formal.
Consequently, by eq.~(\ref{movers}), 
\begin{equation}\label{dimh}
 \dim H^*(M_{b,K}/S) =\dim H^*(M_{b,T}/S).\end{equation}
In the same way, one analyzes the $K'$-action on $M$ and obtains
\begin{equation}\label{movers2}(M/S)^{T'/S}=(M_{b',K'}/S)^{T'/S}=M_{b',T'}/S.\end{equation}
as well as
\begin{equation}\label{dimh2} \dim H^*(M_{b',K'}/S) =\dim H^*(M_{b',T'}/S).\end{equation}

We are now in a position to prove the equivalence stated in the lemma.
First, the $K$-action on $M$ is Cohen-Macaulay if and only if so is the induced $T$-action,
see Prop.~\ref{maxtor} (b). But $H^*_T(M)=H^*_{T/S}(M/S)$ and the $T/S$-action on $M/S$ admits fixed points,
thus the latter Cohen-Macaulay condition is equivalent to: the $T/S$-action on $M/S$ is equivariantly formal.
Equivalently, by eqs.~(\ref{movers}) and (\ref{dimh}),
$$\dim H^*(M/S)= \dim H^*(M_{b,K}/S).$$
In exactly the same way, this time by using eqs.~(\ref{movers2}) and (\ref{dimh2}),
the $K'$-action on $M$ is Cohen-Macaulay if and only if 
$$\dim H^*(M/S)= \dim H^*(M_{b',K'}/S).$$ 
The proof is completed by taking into account eq.~(\ref{mbh}).
\end{proof}

\begin{lemma}\label{kernel} Let $M$ be a closed manifold and $K$ a compact and connected Lie group
that acts smoothly on $M$. Let also $H\subset K$ be the  kernel of the action. 
Then the $K$-action on $M$ is Cohen-Macaulay if and only if so is the $K/H$-action.
\end{lemma} 

\begin{proof} Let $b$ and $M_{b,K}$ be as in the proof of Lemma \ref{prop:orbit}.
Pick maximal tori $T$ and  $T'$ in $K$ and $H$ respectively such that $T'\subset T$.
Again, let $S\subset T$ be a subtorus of rank equal to 
${\rm rank} \ \! T -b$ whose action on $M$ is locally free.
We already noticed that the $K$-action on $M$ is Cohen-Macaulay
if and only if 
$$\dim H^*(M/S)= \dim H^*(M_{b,T}/S).$$
Observe that there exists a maximal torus $T_1 \subset K/H$ along with a covering map
$T/T' \to T_1$. % is a maximal torus in $K/H$, cf.~\cite[Ch.~9, Prop.~2 (c)]{Bou}.
Furthermore, for the $K/H$-action on $M$, the maximal rank of an isotropy subgroup 
is $b_1:=b - {\rm rank} \ \! H.$
Since the intersection $T'\cap S$ is finite,  the 
image $S_1$ of $S$ under the projection $T\to T/T'\to T_1$ is a subtorus of rank equal to
${\rm rank} \ \! S={\rm rank} \ \! T_1 -b_1$ which acts locally freely on $M$. As before, the $K/H$-action on $M$ is Cohen-Macaulay if and only if
$$\dim H^*(M/S_1)= \dim H^*(M_{b_1,T_1}/S_1).$$
But $M/S_1=M/S$ and $M_{b_1,T_1}=M_{b, T}$,
thus the equivalence stated  in the lemma is clear.
\end{proof}

Let us now recall that two isometric actions of two connected compact Lie groups on a  Riemannian manifold are
{\it orbit equivalent} if there is an isometry of the manifold to itself which maps
each  orbit of the first group action to an orbit of the
second one.    
%From Lemma ~\ref{prop:orbit} we deduce  as follows:

\begin{theorem}\label{orbit} Let $M$ be a compact Riemannian manifold 
and $K$  a connected and closed subgroup of 
the isometry group of $M$. Assume that the $K$-action on $M$   
 is orbit equivalent to a Cohen-Macaulay action on $M$.
Then the $K$-action  is Cohen-Macaulay as well. 
\end{theorem}

\begin{proof} 
         Let $G$ be the isometry group of $M$. By hypothesis, there exists a connected and closed subgroup $K'\subset G$ whose canonical action on $M$ is Cohen-Macaulay such that the actions of $K$ and $K'$ on $M$ are orbit equivalent
         (notice that a priori $K'$ might not be contained in $G$; in this case, we mod out the kernel of its action on $M$ and use Lemma 
         \ref{kernel} above).  
     Thus there exists an isometry $f : M \to M$ which maps any $K$-orbit to a  $K'$-orbit. 
      Consider the closed subgroup
     $K''$ of $G$ generated by $f^{-1} K' f$ and $K$. Note that $K''$ is connected.
     The key-observation  is that for any $p\in M$, we have
     \begin{equation}\label{hp} Kp=(f^{-1}K'f)p=K''p.\end{equation}
     The first equality is clear and  immediately implies the second.  
     %To justify  the second, take into account that for any $g\in K'$, the isometry $f^{-1}gf$ transforms an arbitrary orbit $Kp$ as follows: $f$ transforms it to a connected component   of some orbit $K'q$, which in turn is left invariant by $g$, and then brought back by $f^{-1}$ to $Kp$. 
     
     Due to eq.~(\ref{hp}), the result stated in the theorem is  a direct consequence of Lemma \ref{prop:orbit},
     used twice.\end{proof}

We are now in a position to prove the main result of the paper:

\noindent{\it Proof of Thm.~\ref{thm:realmain}.}  
As already mentioned in the introduction, Kollross has shown in  \cite{Ko} that any hyperpolar action is orbit equivalent
  to a direct product of actions of one of the following types:
  transitive, of cohomogeneity one, or Hermann. We know that an action of each of these three types
  is Cohen-Macaulay: for the first two, see \cite[Cor.~1.2]{Go-Ma}, for the last,
 use Thm.~\ref{main}. We apply Thm.~\ref{orbit}: even though the acting group is
  not necessarily a subgroup of ${\rm Iso}(M)$, we can mod out the kernel of the action and use
  Lemma \ref{kernel}.\hfill $\square$ 
  
  \begin{remark} {\rm It would be interesting to find a classification-free proof of Thm.~\ref{thm:realmain}, using the very definition of a hyperpolar action.}   
  \end{remark}
  
  \appendix
  \section{The non-abelian Atiyah-Bredon exact sequence for Cohen-Macaulay actions} 
Although not directly related to the main topic of this paper, the following result is devoted to illustrate the importance of showing that a group action is Cohen-Macaulay. In the special case when the acting group is abelian,
the result was proved by the first-named author of this paper and D.~T\"oben in \cite{Go-To}.
For an arbitrary (compact and connected) acting group, versions of it were obtained
by M.~Franz in \cite{Franz}. We refer to the two papers above for a discussion concerning
the history of the topic, which goes back to M.~Atiyah \cite{Atiyah} and G.~E.~Bredon  \cite{Bredon}.

Let $M$ be a closed manifold equipped with a smooth action of a compact and connected Lie group
$G$ and let $b$ be the maximal rank of a $G$-isotropy subgroup. (Note that $b$ is in general smaller than the rank of $G$.)  
The corresponding orbit filtration is defined by
$$M_i:=\{ p \in M \mid {\rm rank} \ \! G_p \ge i\},$$
where $i$ is arbitrary between $0$ and $b+1$.  
For any $1\le i \le b$, consider the map
$$H_G^*(M_{i}, M_{i+1}) \to H_G^{*+1}(M_{i-1}, M_{i}),$$
which is the connecting homomorphism in the long exact sequence in $G$-equivariant cohomology of the triple $(M_{i-1}, M_i, M_{i+1})$.
By concatenating these maps one obtains the following long sequence:
\begin{equation}\label{ab}0\to H^*_G(M)\to H^*_G(M_b)\to H^{*+1}_G(M_{b-1}, M_b)\to \ldots \to H_G^{*+b-1}(M_{1}, M_{2}) \to H^{*+b}_G(M,M_1)\to 0,
\end{equation}
which we call the Atiyah-Bredon sequence.

\begin{theorem} The $G$-action on $M$ is Cohen-Macaulay if and only if the Atiyah-Bredon sequence  is exact.\end{theorem}

\begin{proof} The result  follows using the same argument as in the proof of \cite[Thm.~6.1]{Go-To},
cf.~also \cite[Sect.~4]{FranzPuppe2003}. The main ingredients in the proof are:

\begin{itemize}
\item The exactness of the sequence (\ref{ab}) is equivalent to the exactness of
 $$0\to H^*_G(M, M_i) \to H^*_G(M_{i-1}, M_i) \to H^{*+1}_G(M, M_{i-1}) \to 0,$$
for any $1 \le i \le b$; in other words, the long exact sequence of the triple $(M_{i-1}, M_i, M_{i+1})$ splits into short exact sequences.
This can be shown going along the same lines as in the proof of \cite[Lemma 4.1]{Fr-Pu}.
\item For any $1\le i \le b$, the Krull dimension of $H^*_G(M, M_i)$ is at most $i-1$. This is the content of \cite[Lemma 4.4]{FranzPuppe2003}
in the case when the acting group is abelian; the arguments of the proof carry over to the general situation, see
\cite[Lemma 2.1]{Go-Ma0} (this result was proved for cohomology with coefficients in
$\bR$, thus it consequently holds for coefficients in $\bQ$, see Remark \ref{pair} above).      
\item For any $1\le i \le b$, $H^*_G(M_{i-1}, M_i)$ is Cohen-Macaulay of dimension $i-1$.
For this, we refer to \cite[Cor.~1]{Franz}.
\end{itemize}

\end{proof}

\end{document}